\documentclass[11pt, letterpaper,reqno]{amsart}

\usepackage{amsmath,amsthm,amsrefs,amssymb,hyperref}
\usepackage[usenames,dvipsnames]{xcolor}
\usepackage{todonotes}
\usepackage{fancyhdr}

\numberwithin{equation}{section}

\newtheorem{corollary}{Corollary}[section]
\newtheorem{lemma}{Lemma}[section]
\newtheorem{theorem}{Theorem}[section]

\theoremstyle{definition}
\newtheorem{remark}{Remark}[section]

\newtheorem{innercustomgeneric}{\customgenericname}
\providecommand{\customgenericname}{}
\newcommand{\newcustomtheorem}[2]{%
	\newenvironment{#1}[1]
	{%
		\renewcommand\customgenericname{#2}%
		\renewcommand\theinnercustomgeneric{##1}%
		\innercustomgeneric
	}
	{\endinnercustomgeneric}
}

\newcustomtheorem{customthm}{Theorem}
\newcustomtheorem{customcorollary}{Corollary}
\newcustomtheorem{customlemma}{Lemma}

\DeclareMathOperator{\D}{\mathbb{D}}

\DeclareMathOperator{\area}{\mathrm{area}}

\begin{document}
\title[Conformal mappings in weighted Bergman spaces]{Geometric characterizations for conformal mappings in weighted Bergman spaces}

\author{Christina Karafyllia}  
\address{Institute for Mathematical Sciences, Stony Brook University, Stony Brook, NY 11794, U.S.A.}
\email{christina.karafyllia@stonybrook.edu}  

\author{Nikolaos Karamanlis}
\address{Department of Mathematics and Applied Mathematics, University of Crete, Heraklion 70013, Crete, Greece}
\email{karamanlisn@uoc.gr} 
\thanks{The second author is supported by the Hellenic Foundation for Research and Innovation, Project HFRI-FM17-1733.}




\begin{abstract} We prove that a conformal mapping defined on the unit disk belongs to a weighted Bergman space if and only if certain integrals involving the harmonic measure converge. With the aid of this theorem, we give a geometric characterization of conformal mappings in Hardy or weighted Bergman spaces by studying Euclidean areas. Applying these results, we prove several consequences for such mappings that extend known results for Hardy spaces to weighted Bergman spaces. Moreover, we  introduce a number which is the analogue of the Hardy number for weighted Bergman spaces. We derive various expressions for this number and hence we establish new results for the Hardy number and the relation between Hardy and weighted Bergman spaces.
\end{abstract}

\maketitle

\section{Introduction}\label{intro}

The weighted Bergman space with exponent $p>0$ and weight $\alpha>-1$ is denoted by $A_\alpha ^p (\D)$ and is defined to be the set of all holomorphic functions $f$ on the unit disk $\mathbb{D}$ such that
\[\left\| f \right\|_{A_\alpha ^p (\D)}^p : = \int_\mathbb{D} {{{\left| {f\left( z \right)} \right|}^p}{{\left( {1 - {{\left| z \right|}^2}} \right)}^\alpha }dA\left( z \right)}  < \infty,
\] 
where $dA$ denotes the Lebesgue area measure on $\mathbb{D}$. The unweighted Bergman space ($\alpha=0$) is simply denoted by $A^p (\D)$ and it is known as the Bergman space with exponent $p$. Weighted Bergman spaces are related to the classical Hardy spaces. The Hardy space with exponent $p>0$ is denoted by $H^p (\D)$ and is the set of all holomorphic functions $f$ on $\D$ such that
\[\left\| f \right\|_{H^p (\D)}^p: = \sup_{0<r<1}\int_{0}^{2\pi} {{{\left| {f\left( re^{it} \right)} \right|}^p} dt}  <  \infty.\]
The Hardy space $H^p (\D)$ is identified with the limit space of $A_\alpha ^p (\D)$,  as $\alpha\to -1^+$, in the sense that $\lim_{\alpha\to -1^+}\left\| f \right\|_{A_\alpha ^p (\D)}= \left\| f \right\|_{H ^p(\D)}$. Moreover, $H^p (\D)\subset A_{\alpha}^p (\D)$, for all $\alpha>-1$ and $p>0$ (see \cite{Zhu}). More on the theory of Bergman spaces can be found in \cite{HKZ} and \cite{DS}.

A classical problem in geometric function theory is to characterize conformal mappings which are contained in such spaces. See, for example, \cite{Ba2}, \cite{Kar1}, \cite{Per}, \cite{Cor}, \cite{Cor2}. In the first part of this paper, we characterize conformal mappings which are contained in $A_\alpha^p(\D)$ or $H^p(\D)$ by using conditions involving certain harmonic measures and Euclidean areas.

First, we fix some notation. For a domain $\Omega$ in the plane, a point $z \in \Omega$ and a Borel subset $A$ of $\overline \Omega$, let ${\omega _{\Omega}}\left( {z,A} \right)$ denote the harmonic measure at $z$ of $A$ with respect to the component of $\Omega \backslash A$ containing $z$. The function ${\omega _{\Omega}}\left( { \cdot ,A} \right)$ is the solution of the generalized Dirichlet problem with boundary data $\varphi  = {1_A}$. For the general theory of harmonic measure, see \cite{Gar}.

Henceforth, let $f$ be a conformal mapping on $\D$. For $r>0$, set $F_r=\{z\in\D:\ |f(z)|=r\}$. Note that $f(F_r)=f(\D)\cap\{|z|=r\}$ is the union of countably many open arcs in $f(\D)$. It follows (see \cite[Prop. 2.14]{Pom2}) that $F_r$ is the union of countably many analytic open arcs in $\D$ so that each such arc has two distinct endpoints on $\partial\D$. Moreover, it is well known (see, for example, \cite{Pom2}) that $f$ has nontangential boundary values $f(e^{it})$ for a.e. real $t$ and thus, for $r>0$, we can define the set $E_r=\{\zeta\in \partial \D:|f(\zeta)|> r\}$.  See Figure \ref{fig}. In \cite{Cor} Poggi-Corradini gave necessary and sufficient conditions for $f$ to belong to $H^p (\D)$ by studying the harmonic measures $\omega_{\D}(0,F_r)$ and $\omega_{\D}(0,E_r)$. He actually proved (see also \cite{Ess2}) that, for $p>0$, $f\in H^p (\D)$ if and only if 
\begin{equation}\label{poco1}
\int_{0}^{\infty} r^{p-1}\omega_{\D} (0,F_r)dr<\infty
\end{equation}
or if and only if
\begin{equation}\label{poco2}
\int_{0}^{\infty} r^{p-1}\omega_{\D} (0,E_r)dr<\infty.
\end{equation}

Furthermore, he observed that the Beurling-Nevanlinna projection theorem implies that for every $r>0$,
\begin{equation}\label{bn}
\omega_{\D}(0,F_r)\ge \frac{2}{\pi}e^{-d_{\D}(0,F_r)},
\end{equation}
where $d_{\D}(0, F_r)$ denotes the hyperbolic distance in $\D$ between $0$ and the set $F_r$, i.e., $d_{\D}(0, F_r)= \inf_{z\in F_r} d_{\D}(0,z)$. Here, $d_{\D}(0,z)$ is the hyperbolic distance between $0$ and $z$ in $\D$. This observation led him to the question whether $\omega_{\D}(0,F_r)$ and $e^{-d_{\D}(0,F_r)}$ are comparable. Moreover, he posed the question whether we could obtain a condition similar to the condition (\ref{poco1}) by replacing the harmonic measure  $\omega_{\D}(0,F_r)$ with the quantity $e^{-d_{\D}(0,F_r)}$. More precisely, he asked whether $f\in H^p(\D)$ if and only if 
\begin{equation}\label{kaha}
\int_{0}^{\infty} r^{p-1}e^{-d_{\D}(0,F_r)}dr<\infty.
\end{equation}
Obviously, if $\omega_{\D}(0,F_r)$ and $e^{-d_{\D}(0,F_r)}$ were comparable, the answer would be positive trivially. However, the first author proved in \cite{Kar2} that these quantities are not comparable in general. So, it was not clear whether the equivalence above is true or not. By applying different methods, the first author showed in \cite{Kar1} that $f\in H^p(\D)$ if and only if \eqref{kaha} is true.

Later, in \cite{Kar} Betsakos and the current authors generalized this condition to weighted Bergman spaces. More specifically, they proved the following theorem. Note that, henceforth, we use the convention that $A^{p}_{-1}(\D)=H^p(\D)$, for every $p>0$, because all the results we state below for weighted Bergman spaces also hold for Hardy spaces.

\begin{customthm}{A}\label{bkk}
\textit{Let $p>0$ and $\alpha\ge-1$. Suppose $f$ is a conformal mapping on $\D$ and, for $r>0$, let $F_r=\{z\in\D:\ |f(z)|=r\}$. Then $f\in A_{\alpha}^p (\D)$ if and only if
\[\int_{0}^{\infty} r^{p-1}e^{-(\alpha+2)d_{\D}\left(0, F_r\right)}dr<\infty.\]
}
\end{customthm} 

Note that the case $\alpha=-1$ is (\ref{kaha}). Therefore, the remaining question is whether we can also extend the conditions (\ref{poco1}) and (\ref{poco2}) to weighted Bergman spaces. In the next section, we prove Theorem \ref{main} which shows that the answer is positive and thus we obtain necessary and sufficient conditions for conformal mappings of $\D$ to belong to $ A_{\alpha}^p (\D)$ by studying the harmonic measure. 

\begin{theorem}\label{main}
Let $f$ be a conformal mapping on $\D$. For $r>0$, we set $F_r=\{z\in\D:\ |f(z)|=r\}$ and $E_r=\{\zeta\in \partial \D: |f(\zeta)|>r\}$. If $p>0$ and $\alpha \ge -1$, the following statements are equivalent.
\begin{enumerate}
\item $f\in A_{\alpha}^p (\D)$,
\item $\displaystyle \int_{0}^{\infty} r^{p-1}\omega_{\D} (0,F_r)^{\alpha+2}dr<\infty$,
\item $\displaystyle \int_{0}^{\infty} r^{p-1}\omega_{\D} (0,E_r)^{\alpha+2}dr<\infty$.
\end{enumerate}

\end{theorem}

Note that the case $\alpha=-1$ is (\ref{poco1}) and (\ref{poco2}). An immediate consequence of Theorem \ref{main} is the following corollary. It is well known (see \cite[Theorem 3.16]{Dur}) for Hardy spaces that if $f$ is conformal on $\D$, then $f\in H^p(\D)$ for all $p\in (0,\frac{1}{2})$. Using Theorem \ref{main}, we can extend this result to weighted Bergman spaces.
\begin{corollary}\label{corol}
If $f$ is a conformal mapping on $\D$, then $f\in A_\alpha^p(\D)$ for every $p>0$ and $\alpha>-1$ such that $\frac{p}{\alpha+2}\in (0,\frac{1}{2})$. Moreover, the inequality $\frac{p}{\alpha+2}<\frac{1}{2}$ is sharp.
\end{corollary}
The case $\alpha=0$ of Corollary \ref{corol} appears in \cite[p. 850]{Ba2}. Combining Theorems \ref{bkk} and \ref{main} we can prove one more characterization of conformal mappings in weighted Bergman spaces (or Hardy spaces) involving this time both the harmonic measure and the hyperbolic distance. 

\begin{corollary}\label{corb}
Let $p>0$ and $\alpha\ge-1$. Suppose $f$ is a conformal mapping on $\D$ and, for $r>0$, set $F_r=\{z\in\D:\ |f(z)|=r\}$. The following statements hold.
\begin{enumerate}
\item If 
\begin{equation}
\int_{0}^{\infty}r^{p-1}e^{-\beta d_{\D}(0,F_r)}\omega_{\D}(0,F_r)^{\gamma}dr<\infty \nonumber
\end{equation}
for some $\beta,\gamma \geq 0$ with $\beta+\gamma=\alpha+2$, then $f\in A_{\alpha}^p(\D)$.
\item If $f\in A_{\alpha}^p(\D)$, then 
\begin{equation}
\int_{0}^{\infty}r^{p-1}e^{-\beta d_{\D}(0,F_r)}\omega_{\D}(0,F_r)^{\gamma}dr<\infty \nonumber
\end{equation}
for any $\beta,\gamma \geq 0$ with $\beta+\gamma=\alpha+2$.
\end{enumerate}
\end{corollary}

With the aid of Corollary \ref{corb}, we establish a Euclidean geometric condition for conformal mappings in $A_\alpha^p(\D)$ (or $H^p(\D)$) involving the area of the set $\{z\in\D:\ |f(z)|>r\}$ for $r>0$. See Figure \ref{fig}.

\begin{theorem}\label{geom}
Let $p>0$ and $\alpha \ge -1$. Suppose that $f$ is a conformal mapping on $\D$ and, for $r>0$, set $U_r=\{z\in\D:\ |f(z)|>r\}$. Then $f\in A_\alpha^p(\D)$ if and only if
\begin{equation}\nonumber
\int_{0}^{\infty}r^{p-1}\area(U_r)^{\frac{\alpha+2}{2}}dr<\infty. 
\end{equation}
\end{theorem}

An interesting consequence of Theorems \ref{bkk}, \ref{main} and \ref{geom} is the following result which, in case $\Phi (r)=\omega_{\D}(0,F_r)$, is an extension of Ess{\' e}n's main lemma in \cite{Ess} for conformal mappings to weighted Bergman spaces.

\begin{figure}\label{fig}
	\includegraphics[width=0.98\linewidth]{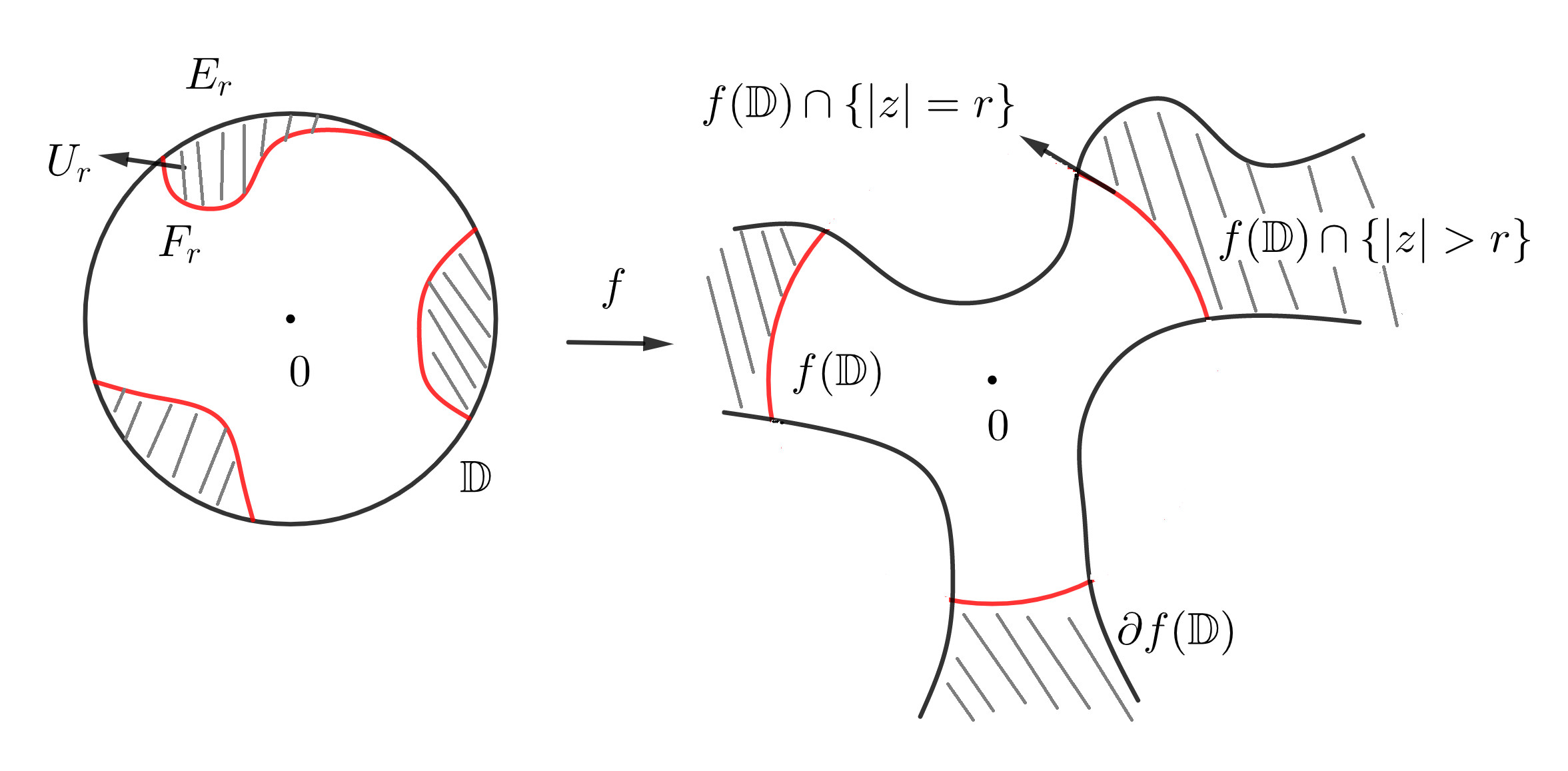}
	\caption{The set $F_r$ is the preimage, under $f$, of $f(\D)\cap\{|z|=r\}$. The  shaded portion of $\D$ is the set $U_r$. It is the preimage of the portion of $f(\D)$ lying outside the red circle of radius $r$. Each component of $\partial U_r$ consists of a component of $F_r$ and a component of $E_r$.}	
\end{figure}

\begin{corollary}\label{l2}
	Let $f$ be a conformal mapping on $\D$. Let $\Phi (r)$ denote $\omega_{\D}(0,F_r)$, $\omega_{\D}(0,E_r)$, $e^{-d_{\D}(0,F_r)}$, or $(\area (U_r))^{1/2}$. If $f\in A_\alpha^p(\D)$ for some $p>0$ and $\alpha\ge-1$, then there is a positive constant $C=C(f,p,\alpha)$ which depends only on $f,p,\alpha$ such that
\[	\Phi (r)\le Cr^{-\frac{p}{\alpha+2}},\]
	for every $r>0$. Moreover, if there are $p'>0$, $\alpha '\ge-1$, $C>0$, and $r_0>0$ such that 
	\[\Phi (r)\le Cr^{-\frac{p'}{\alpha'+2}}\]
	for every $r>r_0$, then $f\in A_\alpha^p(\D)$ for all $p>0$ and $\alpha\geq -1$ such that $\frac{p}{\alpha+2}\in (0,\frac{p'}{\alpha'+2})$. 
\end{corollary}

Next, we define a new notion which is the analogue of the Hardy number for weighted Bergman spaces. We call this the Bergman number. This allows us to establish new results for the Hardy number and the relation between Hardy and weighted Bergman spaces. Furthermore, it provides new ways to determine whether a conformal mapping on $\D$ belongs to $A_\alpha^p(\D)$; this time by studying limits instead of integrals.

In \cite{Han1}, Hansen studied the problem of determining the numbers $p>0$ for which a holomorphic function $f$ on $\mathbb{D}$ belongs to ${H^p}\left( \mathbb{D} \right)$ by studying $f\left( \mathbb{D} \right)$. For this purpose, he introduced a number which he called the Hardy number of a region. Since we are studying conformal mappings, we state the definition just in this case. Let $f$ be a conformal mapping on $\D$. The Hardy number of $f$ is defined by 
\[h(f) = \sup \left\{ {p > 0:f \in {H^p}\left( \mathbb{D} \right)} \right\}.\]
Since $f$ is conformal, according to what we mentioned above, $h(f)$ takes values in $[1/2,\infty]$. The Hardy number has been studied extensively over the years. A classical problem is to find estimates or exact descriptions for it. See, for example, \cite{Han1}, \cite{Kar3} \cite{Kim}, \cite{Cor} and references therein.

In a similar way, we introduce a corresponding number for weighted Bergman spaces. Let $f$ be a conformal mapping on $\D$. We define the Bergman number, $b(f)$, of $f$ as
\[b(f)=\sup\left\{ \frac{p}{\alpha+2}:\ f\in A_{\alpha}^p(\D) \right\}.\]
By Corollary \ref{corol} we deduce that $b(f)\ge 1/2$. Corollary \ref{l2} and the definition above imply the following result.

\begin{corollary}\label{bnd}
Let $p>0$, $\alpha>-1$. If $\frac{p}{\alpha+2}<b(f)$ then $f\in A_{\alpha}^p(\D)$ and if $\frac{p}{\alpha+2}>b(f)$ then $f\notin A_{\alpha}^p(\D)$.	
\end{corollary}

As in the case of the Hardy number, if $\frac{p}{\alpha+2}=b(f)$ then both can happen (see \cite{Kar3}). Using Corollary \ref{l2} and Theorems \ref{bkk}, \ref{main} and \ref{geom} we prove that $b(f)$ is given by the following limits.

\begin{theorem}\label{benu}
Let $f$ be a conformal mapping on $\D$. If $F_r=\{z\in \D : |f(z)|=r \}$, $E_r=\{\zeta\in \partial \D:|f(\zeta)|> r\}$  and $U_r=\{z\in\D :|f(z)|>r\}$, then
\begin{enumerate}
\item $\displaystyle b(f)=\liminf_{r\to \infty}\frac{\log\left(\omega_{\D}(0,F_r)\right)^{-1}}{\log r}$,
\item $\displaystyle b(f)=\liminf_{r\to \infty}\frac{\log\left(\omega_{\D}(0,E_r)\right)^{-1}}{\log r}$,
\item $\displaystyle b(f)=\liminf_{r\to \infty}\frac{d_{\D}(0,F_r)}{\log r}$,
\item $\displaystyle b(f)=\frac{1}{2}\liminf_{r\to \infty}\frac{\log \left(\area(U_r)\right)^{-1}}{\log r}$.
\end{enumerate}
\end{theorem}

By Lemma 3.2 in \cite{Kim} and Theorem 1.1 in \cite{Kar3}, an immediate consequence of Theorem \ref{benu} is that if $f$ is a conformal mapping on $\D$, then 

\begin{equation}\label{habe}
h(f)=b(f). 
\end{equation}

Thus, by Theorem \ref{benu} and (\ref{habe}) we derive new characterizations of the Hardy number in terms of the harmonic measure $\omega_{\D}(0,E_r)$ and the area of $U_r$. That is,
\[h(f)=\liminf_{r\to \infty}\frac{\log\left(\omega_{\D}(0,E_r)\right)^{-1}}{\log r}\]
and
\[h(f)=\frac{1}{2}\liminf_{r\to \infty}\frac{\log \left(\area(U_r)\right)^{-1}}{\log r}.\] 

Moreover, by (\ref{habe}) and Theorem \ref{main}, we can establish the following inclusions between Hardy spaces and weighted Bergman spaces. Let $\mathcal{U}$ denote the set of all conformal mappings on $\D$. By a theorem of Hardy and Littlewood \cite{HL} (also see \cite{Ba2}), for any $p>0$, $H^p(\D)\subset A^{2p}(\D)$ ($\alpha=0$). The second part of the following corollary is an improvement of this fact for the class $\mathcal{U}$.

\begin{corollary}\label{inclu}
Let $p>0$, $\alpha>-1$. We have
\[A_{\alpha}^p(\D)\cap\mathcal{U}\subset H^{q}(\D)\cap\mathcal{U},\]
for every $q\in(0,\frac{p}{\alpha+2})$. Moreover,
\[H^{q}(\D)\cap\mathcal{U}\subset A_{\alpha}^p(\D)\cap\mathcal{U},\]
for every $p>0$, $\alpha>-1$ satisfying $\frac{p}{\alpha+2}\in (0,q]$.
\end{corollary}

In Section \ref{exa} we give an example of a conformal mapping $f$ such that $f\in A_1^3(\D)$ but $f\notin H^1(\D)$, showing that the second inclusion of Corollary \ref{inclu} is, in general, strict. Furthermore, we prove that this conformal mapping $f$ has the property that $f\in A_1^3(\D)$ but $f\notin A_0^2(\D)$. This gives a negative answer to the following natural question. If  $\alpha,\alpha'>-1$ and $p,p'>0$ are such that $\frac{p}{\alpha+2}=\frac{p'}{\alpha'+2}$, is it true that 
\[A_{\alpha}^p(\D)\cap\mathcal{U}=A_{\alpha'}^{p'}(\D)\cap\mathcal{U}?\]

We now discuss a different characterization of conformal mappings in weighted Bergman spaces. For $t\in [0,1)$, let $M_f(t)=M(t):=\max_{|z|=t}|f(z)|$. We have the following result.
\begin{customthm}{B}\label{bgp}
\textit{Let $p>0$, $\alpha\geq -1$ and suppose $f$ is a conformal mapping on $\D$. Then $f\in A_{\alpha}^p(\D)$ if and only if
\begin{equation}\label{growth}
\int_{0}^{1}(1-t)^{\alpha+1}M(t)^pdt<\infty.
\end{equation}
}
\end{customthm}

The case $\alpha=-1$ is a statement about Hardy spaces and it is due to Hardy and Littlewood \cite{HL}, Pommerenke \cite{Pom1}, and Prawitz \cite{Pra}. The case $\alpha>-1$, is due to Baernstein, Girela and Pel{\'a}ez \cite{Ba2} and also to P{\'e}rez-Gonz{\'a}lez and R{\"a}tty{\"a} \cite{Per}. Using Theorem \ref{bgp} and a geometric characterization for conformal mappings in weighted Bergman spaces which appears in \cite[relation (2.9)]{Per}, we can prove the following corollary for conformal mappings in $A_{\alpha}^p(\D)$ (or $H^p(\D)$). Let $D(0,r)$ denote the disk centered at $0$ of radius $r$. If $f$ is a conformal mapping on $\D$, we set  $\psi(r)=\area\left(f(D(0,r))\right)$, for $r\in (0,1)$.

\begin{corollary}\label{g1}
Let $p>0$, $\alpha\geq-1$ and suppose $f\in A_{\alpha}^p(\D)$ is a conformal mapping. Then
\begin{enumerate}
\item $\displaystyle \lim_{r\to 1}(1-r)M(r)^{\frac{p}{\alpha+2}}=0$,
\item $\displaystyle \lim_{r\to 1}(1-r)\psi(r)^{\frac{p}{2(\alpha+2)}}=0$.
\end{enumerate}
\end{corollary}

Theorem \ref{bgp}, Corollary \ref{g1}, and the aforementioned geometric characterization in \cite{Per}, allow us to provide the following expressions for the Bergman number, and thus for the Hardy number as well, in terms of the functions $M_f$ and $\psi$.
\begin{theorem}\label{g2}
Suppose $f$ is a conformal mapping in $\D$. Then
\begin{enumerate}
\item $\displaystyle b(f)=h(f)=\liminf_{r\to 1}\frac{-\log(1-r)}{\log M(r)}$,
\item $\displaystyle b(f)=h(f)=2\liminf_{r\to 1}\frac{-\log(1-r)}{\log \psi(r)}$,
\item $\displaystyle b(f)=h(f)=\sup\{\lambda>0:\ \lim_{r\to 1}(1-r)M(r)^{\lambda}=0\}$,
\item $\displaystyle b(f)=h(f)=2\sup\{\lambda>0:\ \lim_{r\to 1}(1-r)\psi(r)^{\lambda}=0\}$.
\end{enumerate}
\end{theorem}

The proofs of all the results above follow in Section \ref{proofs}. 

\section{Proofs}\label{proofs}
\subsection{Proof of Theorem \ref{main}}

The main tools we use in the proof of Theorem \ref{main} include Smith's characterization of functions belonging to weighted Bergman spaces \cite{Smith}, a potential theoretic identity of Baernstein connecting harmonic measure and the Green function \cite{Ba1} and an estimate for harmonic measure due to Poggi-Corradini \cite{Cor2}.

\proof
Let $f$ be a conformal mapping defined on $\D$ and for $r>0$, let $F_r$ and $E_r$ be as in the statement of the theorem. Set $D=f(\D)$. By comparing boundary values and using the maximum principle, we have that
\[\omega_{\D}(0,E_r)\le \omega_{\D}(0,F_r),\]
for any $r>0$. Moreover, it is proved in \cite[p. 34]{Cor} (see also \cite[Lemma 3.3(i)]{Cor2}) that there exists a constant $M>1$ such that
\begin{equation}\label{hmest}
\omega_{\D}(0,F_{Mr})\le 2\omega_{\D}(0,E_r),
\end{equation}
for any $r>|f(0)|$. Therefore, combining these two estimates and using a change of variable, we can easily show that (2) and (3) are equivalent.

Suppose now that (2) holds, i.e.,
\[\int_{0}^{\infty} r^{p-1}\omega_{\D} (0,F_r)^{\alpha+2}dr<\infty.\]
Theorem \ref{bkk} and \eqref{bn} imply directly that $f\in A_{\alpha}^p (\D)$ and thus (2) implies (1). Conversely, assume that (1) holds. Since (2) and (3) are equivalent, in order to complete the proof of the theorem, it suffices to prove that (3) also holds. Since $f\in A_\alpha^p(\D)$, it follows (see \cite[p. 2336]{Smith}) that
\begin{equation}\label{finite}
\int_{\D} {|f(z)|^{p - 2} | f '(z)|^2 \left(\log \frac{1}{|z|}\right)^{\alpha+2}dA(z)}<\infty.
\end{equation}
For the Green function (see \cite{Gar} for the definition) of the domain $D$, we set $g_D(f(0),w)=0$, for $w\notin D$. By a change of variable and the conformal invariance of the Green function, we have

\begin{align}\label{equal}
\int_{\D} |f(z)|^{p - 2} | f '(z)|^2 & \left( \log \frac{1}{|z|}\right)^{\alpha+2}dA(z) \nonumber \\
&=\int_{\D} {|f(z)|^{p - 2} | f '(z)|^2 g_{\D}(0,z)^{\alpha+2}dA(z)} \nonumber \\
&=\int_D {|w|^{p - 2} g_{\D}(0,f^{-1}(w))^{\alpha+2}dA(w)} \nonumber \\
&= \int_D {|w|^{p - 2} g_{D}(f(0),w)^{\alpha+2}dA(w)}   \nonumber \\
&=\int_0^{\infty } {r ^{p - 1}\left( {\int_0^{2\pi } {{g_D}( f(0),re^{i\theta })^{\alpha+2}d\theta } } \right)dr }.
\end{align}
Since $\alpha+2>1$, by Jensen's inequality we derive that, for every $r>0$,
\begin{equation}
\left( \frac{1}{2\pi} \int_0^{2\pi} {g_D(f(0),re^{i\theta})d\theta} \right)^{\alpha+2} \le\frac{1}{2\pi} \int_0^{2\pi} {{g_D}( f(0),re^{i\theta })^{\alpha+2}d\theta}. \nonumber
\end{equation}
This in conjunction with (\ref{finite}) and (\ref{equal}) implies that
\begin{equation}\label{doubleint}
\int_0^{ \infty } {r ^{p - 1}\left( {\int_0^{2\pi } {{g_D}( f(0),re^{i\theta })d\theta } } \right)^{\alpha+2}dr }<\infty.
\end{equation}
Now we state a known relation between harmonic measure and the Green function. For the proof, see \cite[Lemma 2]{Ba1}. If $\Omega \subset \mathbb{C}$ is a simply connected domain and $a\in \Omega$ then, for $r>|a|$, it is true that
\begin{equation}\label{baernid}
\int_r^{\infty}\omega_{\Omega}\left({a,\{|z|>t\}}\cap\partial\Omega\right)\frac{dt}{t}=\frac{1}{2\pi}\int_0^{2\pi} g_{\Omega}({a,re^{i\theta})}d\theta.   
\end{equation}
By \eqref{baernid} and the conformal invariance of the harmonic measure, for every $r>|f(0)|$,
\begin{align}\label{baern}
\int_0^{2\pi} g_D(f(0),re^{i\theta})d\theta &=2\pi\int_r^{\infty} \omega_D (f(0),f(E_t))\frac{dt}{t}=2\pi\int_r^{\infty} \omega_{\D}(0,E_t)\frac{dt}{t} \nonumber \\
&\ge 2\pi \int_r^{2r} \omega_{\D}(0,E_t)\frac{dt}{t} \ge 2\pi \log 2\, \omega_{\D}(0,E_{2r}).
\end{align}
In the last estimate of \eqref{baern}, we have used the fact that $\omega_{\D}(0,E_t)$ is decreasing for $t>0$ because of the maximum principle. Therefore, by (\ref{baern}) we have
\[  \int_{|f(0)|}^{\infty} r^{p-1} \omega_{\D}(0,E_{2r})^{\alpha+2}dr \le C_1 \int_{|f(0)|}^{\infty} r^{p-1} \left( {\int_0^{2\pi } {{g_D}( f(0),re^{i\theta })d\theta } } \right)^{\alpha+2}dr,\]
where $C_1=1/(2\pi\log 2)^{\alpha+2}$. Equivalently, we can write
\[  \int_{2|f(0)|}^{\infty} r^{p-1} \omega_{\D}(0,E_r)^{\alpha+2}dr \le C_2 \int_{|f(0)|}^{\infty} r^{p-1} \left( {\int_0^{2\pi } {{g_D}( f(0),re^{i\theta })d\theta } } \right)^{\alpha+2}dr, \]
where $C_2=2^p/(2\pi\log 2)^{\alpha+2}$. Combining this with (\ref{doubleint}), we obtain
\[\int_{2|f(0)|}^{\infty} r^{p-1} \omega_{\D}(0,E_r)^{\alpha+2}dr<\infty \]
and thus
\[\int_0^{\infty} r^{p-1} \omega_{\D}(0,E_r)^{\alpha+2}dr<\infty.\]
This shows that (3) holds and the proof is complete.
\qed
\subsection{Proof of Corollary \ref{corol}}

By the definition of $A_\alpha^p(\D)$ we can directly derive that it contains any bounded conformal mapping on $\D$, for any $a>-1$ and $p>0$. Hence, it suffices to consider unbounded conformal mappings on $\D$. Let  $f$ be an unbounded conformal mapping on $\D$. By \cite[Lemma 3.3(iii)]{Cor2}, there is a constant $C>0$ such that, for $r>|f(0)|$,
\[\omega_{\D}(0,F_r)\leq \frac{C}{\sqrt{r}}.\] 
Therefore, we have
\[\int_{|f(0)|}^{\infty}r^{p-1}\omega_{\D}(0,F_r)^{\alpha+2}dr\leq C^{\alpha+2}\int_{|f(0)|}^{\infty}r^{p-\frac{\alpha}{2}-2}dr<\infty,\]
for every $p>0$ and $\alpha>-1$ such that $\frac{p}{\alpha+2}\in(0,\frac{1}{2})$. Thus, Theorem \ref{main} implies that $f\in A_\alpha^p(\D)$ for every $p>0$ and $\alpha>-1$ such that $\frac{p}{\alpha+2}\in(0,\frac{1}{2})$. 

Now, we show that the inequality $\frac{p}{\alpha+2}<\frac{1}{2}$ is sharp. Consider the Koebe function $K(z)=\frac{z}{(1-z)^2}$. Using the conformal invariance of the harmonic measure, we deduce that for $r>1/4$,
\[\omega_{\D}(0,F_r)=1-\frac{2}{\pi}\arctan\left(\sqrt{r}\left(1-\frac{1}{4r}\right)\right)\]
and hence, by elementary calculus, if $r$ is large enough, there is a constant $C>0$ such that
\[\omega_{\D}(0,F_r)\geq \frac{C}{\sqrt{r}}.\]
Therefore, for any $\delta$ sufficiently large, if $\frac{p}{\alpha+2}=\frac{1}{2}$, then
\[\int_{\delta}^{\infty}r^{p-1}\omega_{\D}(0,F_r)^{\alpha+2}dr\geq C^{\alpha+2}\int_{\delta}^{\infty}r^{p-\frac{\alpha}{2}-2}dr=C^{\alpha+2} \int_{\delta}^{\infty}\frac{1}{r}dr=\infty.\]
By Theorem \ref{main} it follows that $K\notin A_\alpha^p(\D)$ and the proof is complete.
\qed
\subsection{Proof of Corollary \ref{corb}}

We first prove (1). Let $f$ be a conformal mapping on $\D$ and, for $r>0$, let $F_r$ be as in the statement of Corollary \ref{corb}. Suppose $\beta,\gamma \geq 0$ satisfy $\beta+\gamma=\alpha+2$ and
\[
\int_{0}^{\infty}r^{p-1}e^{-\beta d_{\D}(0,F_r)}\omega_{\D}(0,F_r)^{\gamma}dr<\infty.
\]
By \eqref{bn} and since $\beta+\gamma=\alpha+2$, we immediately find that
\[
\int_{0}^{\infty} r^{p-1}e^{-(\alpha+2)d_{\D}\left(0, F_r\right)}dr<\infty
\]
and thus, by Theorem \ref{bkk}, $f\in A_\alpha^p(\D)$.

We now proceed with the proof of (2). Suppose $f\in A_\alpha^p(\D)$ is a conformal mapping and let $\beta ,\gamma>0$ satisfy $\beta+\gamma=\alpha+2$. Note that the cases $(\beta,\gamma)=(0,\alpha+2)$ and $(\beta,\gamma)=(\alpha+2,0)$ are covered by Theorems  \ref{main} and \ref{bkk}, respectively. If $t=\frac{\alpha+2}{\beta}$ and $s=\frac{\alpha+2}{\gamma}$, then $\frac{1}{t}+\frac{1}{s}=1$. By H{\"o}lder's inequality, Theorem \ref{bkk} and Theorem \ref{main}, we have
\begin{align*}
\int_{0}^{\infty}&r^{p-1}e^{-\beta d_{\D}(0,F_r)}\omega_{\D}(0,F_r)^{\gamma}dr\\
&=\int_{0}^{\infty}r^{(p-1)/t}e^{-\beta d_{\D}(0,F_r)}r^{(p-1)/s}\omega_{\D}(0,F_r)^{\gamma}dr\\
&\leq \left(\int_{0}^{\infty}r^{p-1}e^{-(\alpha+2)d_{\D}(0,F_r)}dr\right)^{1/t}\left(\int_{0}^{\infty}r^{p-1}\omega_{\D}(0,F_r)^{\alpha+2}dr\right)^{1/s} \\
&<\infty.
\end{align*}
and the proof is complete.
\qed
\subsection{Proof of Theorem \ref{geom}}

For the proof of Theorem \ref{geom}, we need the following lemma. Suppose $f$ is a conformal mapping on $\D$ and, for $r>0$, let $F_r$ and $E_r$ be as defined in Section \ref{intro}. We also set $U_r=\{z\in\D:\ |f(z)|>r\}$. Note that for each $r>0$, the open set $U_r$ consists of (at most) countably many components each having the property that its boundary consists of a component of $F_r$ and a component of $E_r$. See Figure \ref{fig}.
\begin{lemma}\label{LA}
There is a universal constant $C>0$ such that if $r>0$ is sufficiently large, then
\[
\area(U_r)\geq C\omega_{\D}(0,E_{2r})^2.
\]

\end{lemma}
\begin{figure}
	\includegraphics[width=0.5\linewidth]{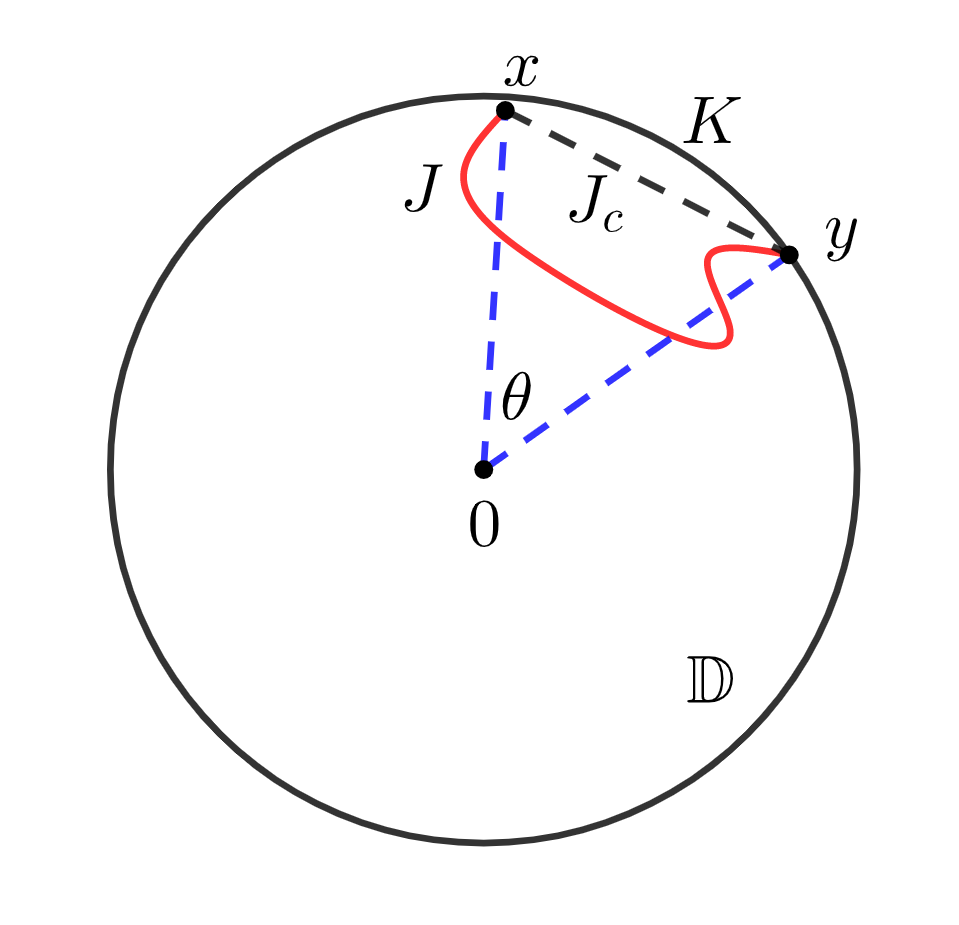}
	\caption{The red arc $J$ is a component of $F_r$ with endpoints $x,y$. Its length is at least as large as the length of the dotted chord $J_c$. The angle $\theta$ is equal to $2\pi\omega_{\D}(0,K)$, where $K$ is the component of $E_r$ with endpoints $x,y$.}	
\end{figure}
\proof
Let $D=f(\D)$ and for $r>0$, set $I_r=\{\theta\in [0,2\pi):\ re^{i\theta}\in D\}$. Using a change of variable, we have
\begin{align*}
\area(U_r)=\int_{U_r}dA(z)&=\int_{D\cap\{|z|>r\}}\frac{dA(w)}{|f'\left(f^{-1}(w)\right)|^2}\\
&= \int_{r}^{\infty}\int_{I_t}\frac{t}{|f'\left(f^{-1}(te^{i\theta})\right)|^2}d\theta dt\\
&= \int_{r}^{\infty}\frac{1}{t}\int_{I_t}\frac{t^2}{|f'\left(f^{-1}(te^{i\theta})\right)|^2}d\theta dt.
\end{align*}
Note that for $t>0$, we may parametrize the curve $F_t$ as $F_t(\theta)=f^{-1}\left(te^{i\theta}\right)$, for $\theta\in I_t$. Using the Cauchy-Schwarz inequality, we infer that
\begin{align*}
\text{length}(F_t)^2 &=\left(\int_{F_t}|dz|\right)^2=\left(\int_{I_t}\frac{t}{|f'\left(f^{-1}(te^{i\theta})\right)|}d\theta\right)^2\\
&\leq \text{length}\left(I_t\right)\int_{I_t}\frac{t^2}{|f'\left(f^{-1}(te^{i\theta})\right)|^2}d\theta\\
&\leq 2\pi\int_{I_t}\frac{t^2}{|f'\left(f^{-1}(te^{i\theta})\right)|^2}d\theta.
\end{align*}
Therefore, by the calculation above, we obtain
\[
\area(U_r)\geq \frac{1}{2\pi}\int_{r}^{\infty}\frac{\text{length}(F_t)^2}{t}dt.
\]
If  $r$ is large and $J$ denotes any component of $F_r$, then $J$ is an analytic arc in $\D$ having two distinct endpoints, $x,y$, on $\partial\D$. Let $J_c$ denote the chord of the disk connecting $x,y$. If $r$ is sufficiently large, then the arc of the circle with endpoints $x,y$ and having smaller length, is a component $K$ of $E_r$. See Figure 2. By elementary geometry and since $\frac{\sin t}{t}\ge \frac{1}{2}$, for $t>0$ sufficiently small,
\[
\text{length}(J)\geq\text{length}(J_c)=2\sin \left(\pi\omega_{\D}(0,K)\right) \ge \pi\omega_{\D}(0,K).
\]
Taking the sum over all the components of $F_r$ and $E_r$ we deduce that
\[
\text{length}(F_t)\geq \pi\omega_{\D}(0,E_t),
\]
for all $t>r$, provided that $r$ is sufficiently large. Hence, 
\begin{align*}
\area(U_r)&\geq \frac{\pi}{2}\int_{r}^{\infty}\frac{\omega_{\D}(0,E_t)^2}{t}dt
\geq\frac{\pi}{2}\int_{r}^{2r}\frac{\omega_{\D}(0,E_t)^2}{t}dt\\
&\geq \frac{\pi}{2}\omega_{\D}(0,E_{2r})^2\int_{r}^{2r}\frac{1}{t}dt
=\frac{\pi\log 2}{2}\omega_{\D}(0,E_{2r})^2,
\end{align*}
for $r$ sufficiently large and the proof is complete.
\qed\\[0.5cm]
\begin{remark}
With a slight modification of the proof of Lemma \ref{LA}, it is possible to show that for each $M>1$, there exists a constant $C=C(M)>0$ which depends only on $M$ such that for all sufficiently large $r$, 
\[
\area(U_r)\geq C\omega_{\D}(0,E_{Mr})^2.
\]
However, for our purposes, the case $M=2$ suffices.
\end{remark}
We can now prove Theorem \ref{geom}.

\proof Let $p>0$, $\alpha \ge -1$ and suppose $f\in A_\alpha^p(\D)$ is a conformal mapping. There is a universal constant $C>0$ such that if $r>0$, then
\begin{equation}\label{marsmith}
\area(U_r)\leq Ce^{-d_{\D}(0,F_r)}\omega_{\D}(0,F_r).
\end{equation}
For a proof of this fact, see \cite[section 2.1]{Mar}. Applying Corollary \ref{corb} (2) with $\beta=\gamma=\frac{\alpha+2}{2}$ we have
\[
\int_{0}^{\infty}r^{p-1}e^{-\frac{\alpha+2}{2}d_{\D}(0,F_r)}\omega_{\D}(0,F_r)^\frac{\alpha+2}{2}dr<\infty.
\]
This in combination with \eqref{marsmith} implies that
\begin{equation}
\int_{0}^{\infty}r^{p-1}\area(U_r)^{\frac{\alpha+2}{2}}dr<\infty. \nonumber
\end{equation}

Conversely, let $p>0$, $\alpha \ge-1$ and assume that $f$ is a conformal mapping on $\D$ satisfying 
\begin{equation}\label{area}
\int_{0}^{\infty}r^{p-1}\area(U_r)^{\frac{\alpha+2}{2}}dr<\infty. 
\end{equation}
By Lemma \ref{LA}, there is some universal constant $C>0$ such that
\begin{equation}\label{el}
\area(U_r)\geq C\omega_{\D}(0,E_{2r})^2,
\end{equation}
for all $r$ sufficiently large. Let $\delta$ be large enough so that if $r>\delta$, then \eqref{el} holds. Then, by a change of variable and \eqref{el},
\begin{align*}
\int_{\delta}^{\infty}r^{p-1}\area(U_r)^{\frac{\alpha+2}{2}}dr &\geq C^{\frac{\alpha+2}{2}}\int_{\delta}^{\infty}r^{p-1}\omega_{\D}(0,E_{2r})^{\alpha+2}dr \\
&= C_1\int_{2\delta}^{\infty}r^{p-1}\omega_{\D}(0,E_r)^{\alpha+2}dr,
\end{align*}
where $C_1={C^{\frac{\alpha+2}{2}}}/{2^p}$. This in combination with (\ref{area}) shows that
\[
\int_{0}^{\infty}r^{p-1}\omega_{\D}(0,E_r)^{\alpha+2}dr<\infty,
\]
and thus by Theorem \ref{main}, $f\in A_\alpha^p(\D)$.
\qed

\subsection{Proof of Corollary \ref{l2}}

Let $f\in A_{\alpha}^p(\D)$ be conformal. Suppose that $\Phi (r)=\omega_{\D}(0,F_r)$. By Theorem \ref{main} we have  
	\[\int_{0}^{\infty}r^{p-1}\omega_{\D}(0,F_r)^{\alpha+2}dr<\infty.
	\]
	Since $\omega_{\D}(0,F_r)$ is a decreasing function of $r$, it follows that for $R>0$,
	\begin{align}
	\int_{0}^{\infty}r^{p-1}\omega_{\D}(0,F_r)^{\alpha+2}dr & \geq\int_{0}^{R}r^{p-1}\omega_{\D}(0,F_r)^{\alpha+2}dr \nonumber\\
	& \geq\omega_{\D}(0,F_R)^{\alpha+2}\int_{0}^{R}r^{p-1}dr=\frac{R^p}{p}\omega_{\D}(0,F_R)^{\alpha+2}. \nonumber
	\end{align}
	Combining the results above we infer that, for every $R>0$,
	\[\omega_{\D}(0,F_R)\leq CR^{-\frac{p}{\alpha+2}},\]
	where 
	\[C=\left(p\int_{0}^{\infty}r^{p-1}\omega_{\D}(0,F_r)^{\alpha+2}dr\right)^{\frac{1}{\alpha+2}}\]
	is a constant which depends only on $f,p$, and $\alpha$. This completes the proof of the first part of the theorem. 
	
	Now, suppose there are $p'>0$, $\alpha '\geq -1$, $C>0$, and $r_0>0$ such that 
	\[\omega_{\D}(0,F_r)\leq Cr^{-\frac{p'}{\alpha'+2}}\]
	for every $r>r_0$. If $\alpha\geq -1$ and $p>0$ satisfy $\frac{p}{\alpha+2}<\frac{p'}{\alpha'+2}$, then it follows that
	\[\int_{r_0}^{\infty}r^{p-1}\omega_{\D}(0,F_r)^{\alpha+2}dr\le C^{\alpha+2}\int_{r_0}^{\infty}r^{p-1-\frac{p'}{\alpha'+2}(\alpha+2)}dr<\infty.\]
	Therefore, Theorem \ref{main} implies that $f\in A_{\alpha}^p(\D)$ for all $p>0$ and $\alpha \geq -1$ such that $\frac{p}{\alpha+2}\in (0,\frac{p'}{\alpha'+2})$. 

Since $\omega_{\D}(0,E_r)$, $e^{-d_{\D}(0,F_r)}$ and $\left(\area (U_r)\right)^{1/2}$ are decreasing functions of $r$, by Theorems \ref{main}, \ref{bkk} and \ref{geom}, respectively, we see that the proof above also works if $\Phi (r)$ is equal to $\omega_{\D}(0,E_r)$, $e^{-d_{\D}(0,F_r)}$ or $\left(\area (U_r)\right)^{1/2}$ and thus the proof is complete.
\qed

\subsection{Proof of Corollary \ref{bnd}}

Let $p>0$, $\alpha>-1$ be such that $\frac{p}{\alpha+2}<b(f)$. Then we may find $\epsilon>0$ so that $\frac{p}{\alpha+2}+\epsilon<b(f)$. Thus there exist $p'>0$ and $\alpha'>-1$ satisfying $\frac{p}{\alpha+2}+\epsilon<\frac{p'}{\alpha'+2}$ and $f\in A_{\alpha'}^{p'}$. By the first part of Corollary \ref{l2}, \[\Phi(r) \leq Cr^{-\frac{p'}{\alpha'+2}},\] for all $r>0$ and some constant $C>0$. Therefore, 
\[\Phi(r)\leq Cr^{-\frac{p+\epsilon(\alpha+2)}{\alpha+2}},\]for all $r>1$. By the second part of Corollary \ref{l2}, it follows that $f\in A_{\beta}^q(\D)$, for all $q>0$, $\beta>-1$ such that $\frac{q}{\beta+2}<\frac{p+\epsilon(\alpha+2)}{\alpha+2}$. Choosing $\beta=\alpha$ and $q=p$ yields $f\in A_{\alpha}^p(\D)$. If $\frac{p}{\alpha+2}>b(f)$, then we clearly have $f\notin A_{\alpha}^p(\D)$. 
\qed

\subsection{Proof of Theorem \ref{benu}}

First, we prove (1). If $f\in A_\alpha^p(\D)$, then by Corollary \ref{l2} there is a constant $C>0$ such that
	\[\omega_{\D}(0,F_r)\le Cr^{-\frac{p}{\alpha+2}},\]
	for every $r>0$. Equivalently, for $r>1$,
	\[\frac{\log \omega_{\D}(0,F_r)^{-1}}{\log r}\ge \frac{\log C^{-1}}{\log r}+ \frac{p}{\alpha+2}.\]
	Thus, taking limits as $r\to \infty$, we deduce that
	\[\liminf_{r\to \infty}\frac{\log\omega_{\D}(0,F_r)^{-1}}{\log r}\ge\frac{p}{\alpha+2}.\]
	This holds for any $p>0$ and $\alpha>-1$ for which $f\in A_\alpha^p(\D)$ and hence
	\begin{align}\label{miaf}
	\liminf_{r\to \infty}\frac{\log\omega_{\D}(0,F_r)^{-1}}{\log r}\ge b(f).
	\end{align}
	Now, we set 
	\[I:=\liminf_{r\to \infty}\frac{\log\omega_{\D}(0,F_r)^{-1}}{\log r}.\]
	If $\frac{p}{a+2}<I$, then there exist $\epsilon>0$ and $r_0>0$ such that for every $r>r_0$,
	\[
	\frac{p}{a+2}+\epsilon\le\frac{\log\omega_{\D}(0,F_r)^{-1}}{\log r}
	\]
	or, equivalently,
	\[r^{p-1}\omega_{\D}(0,F_r)^{\alpha+2}\leq r^{-1-\epsilon(\alpha+2)}.\]
	Therefore, it follows that
	\[
	\int_{r_0}^{\infty}r^{p-1}\omega_{\D}(0,F_r)^{\alpha+2}dr \le \int_{r_0}^{\infty}r^{-1-\epsilon(\alpha+2)}dr <\infty.
	\]
	Theorem \ref{main} implies that $f\in A_{\alpha}^p(\D)$. This shows that the 	   interval $(0,I)$ is contained in the set
	\[
	\left\{\frac{p}{\alpha+2}:\ f\in A_{\alpha}^p(\D)\right\}
	\]	
and hence
	\begin{align}\nonumber
	\liminf_{r\to \infty}\frac{\log\omega_{\D}(0,F_r)^{-1}}{\log r}\le b(f).
	\end{align}
	This in conjunction with (\ref{miaf}) gives the desired result.
	
	For the proof of (2), (3) and (4), it suffices to replace $\omega_{\D}(0,F_r)$ with $\omega_{\D}(0,E_r)$, $e^{-d_{\D}(0,F_r)}$ and $\left(\area (U_r)\right)^{1/2}$, use Theorems \ref{main}, \ref{bkk} and \ref{geom}, respectively, and repeat the proof of (1).
\qed

\subsection{Proof of Corollary \ref{inclu}}

Let $f\in A_{\alpha}^p(\D)\cap\mathcal{U}$ for some $p>0$ and $\alpha>-1$. Then $\frac{p}{\alpha+2}\leq b(f)$. By (\ref{habe}) we infer that $\frac{p}{\alpha+2}\leq h(f)$ and hence $f\in H^{q}(\D)$ for every $q\in(0,\frac{p}{\alpha+2})$.

Now, let $f \in H^{q}(\D)\cap\mathcal{U}$ for some $q>0$. By \eqref{poco1} it follows that
\[\int_{0}^{\infty}r^{q-1}\omega_{\D}(0,F_r)dr<\infty.\]
Moreover, by Corollary \ref{l2} (or \cite[Lemma 1]{Ess}), there is a constant $C>0$ such that 
\[\omega_{\D}(0,F_r)\le Cr^{-q},\]
for every $r>0$. Therefore, if $p>0$ and $\alpha>-1$ satisfy $\frac{p}{\alpha+2}\leq q$, then
\begin{align*}
\int_{1}^{\infty}r^{p-1}\omega_{\D}(0,F_r)^{\alpha+2}dr&=\int_{1}^{\infty}r^{p-1}\omega_{\D}(0,F_r)^{\alpha+1}\omega_{\D}(0,F_r)dr\\
&\le C^{\alpha+1}\int_{1}^{\infty}r^{p-1}r^{-q(\alpha+1)}\omega_{\D}(0,F_r)dr \\
&\le C^{\alpha+1}\int_{1}^{\infty}r^{q-1}\omega_{\D}(0,F_r)dr<\infty.
\end{align*}
So, Theorem \ref{main} implies that $f\in A_{\alpha}^p(\D)\cap\mathcal{U}$ for any $p>0$ and $\alpha>-1$ such that $\frac{p}{\alpha+2}\in (0,q]$.
\qed

\subsection{Proof of Corollary \ref{g1}}
Let $p>0$, $\alpha\geq -1$ and suppose that $f\in A_{\alpha}^p(\D)$ is conformal. We first prove (1). Note that $M_f$ is an increasing function of $r$. Hence, for $r\in (0,1)$,
\[
\int_{r}^{1}(1-t)^{\alpha+1}M(t)^pdt\geq M(r)^p\int_{r}^{1}(1-t)^{\alpha+1}dt=\frac{1}{{\alpha+2}}{(1-r)^{\alpha+2}M(r)^p}.
\]
By Theorem \ref{bgp}, we have 
\[
\lim_{r\to 1} \int_{r}^{1}(1-t)^{\alpha+1}M(t)^pdt= 0.
\]
Combining the above, we deduce that 
\[\lim_{r\to 1}(1-r)^{\alpha+2}M(r)^{p}=0\]
or, equivalently,
\[
\lim_{r\to 1}(1-r)M(r)^{\frac{p}{\alpha+2}}=0\]
which completes the proof of (1).

We now prove (2). For $r\in(0,1)$, let $\psi(r)=\area\left(f(D(0,r))\right)$, where $D(0,r)$ is the disk centered at $0$ of radius $r$. By a result which appears in \cite[(2.9) on p. 132]{Per}, we have that 
\begin{equation}\label{geom29}
\int_{0}^{1}(1-t)^{\alpha+1}\psi(t)^{\frac{p}{2}}dt<\infty.
\end{equation}
Since $\psi$ is an increasing function of $r$, repeating the argument above and using \eqref{geom29} instead of Theorem \ref{bgp} we get (2).
\qed
\subsection{Proof of Theorem \ref{g2}}
Let $f$ be a conformal mapping on $\D$. We first prove (1). Set 
\[
L_f:=\liminf_{r\to 1}\frac{-\log(1-r)}{\log M(r)}.
\]
Let $p>0$, $\alpha >-1$ and assume that $f\in A_{\alpha}^p(\D)$. By (1) of Corollary \ref{g1}, we have 
\[
\lim_{r\to 1}(1-r)M(r)^{\frac{p}{\alpha+2}}=0,
\]
which in turn implies that, for $r$ sufficiently close to $1$,
\[(1-r)M(r)^{\frac{p}{\alpha+2}}\le 1 \]
and thus
\[
\frac{p}{\alpha+2}\leq \liminf_{r\to 1}\frac{-\log(1-r)}{\log M(r)}.
\]
Taking the supremum over all $\frac{p}{\alpha+2}$ for which $f\in A_{\alpha}^p$ yields
\begin{equation}\label{equ1}
b(f)\leq L_f.
\end{equation}
Note that by the Koebe distortion theorem, or by Corollary \ref{corol}, $L_f\geq 1/2$. Choose any $\mu\in(0,L_f)$. Then $\mu<\frac{-\log(1-r)}{\log M(r)}$, for all $r$ sufficiently close to $1$ and thus
\[
M(r)^{\mu}<\frac{1}{1-r},
\]
for $r$ close to $1$. This implies that, if $\delta \in (0,1)$ is sufficiently close to $1$, then
\[
\int_{\delta}^{1}M(r)^{\mu(1-\epsilon)}dr \le \int_{\delta}^{1} (1-r)^{\epsilon-1}<\infty,
\]
for all $\epsilon\in (0,1)$. Therefore, by Theorem \ref{bgp}, $f\in H^{\mu(1-\epsilon)}(\D)$ and hence
\[
h(f)\geq \mu(1-\epsilon),
\]
for all $\epsilon\in (0,1)$. Letting $\epsilon\to 0$ and using the fact that $h(f)=b(f)$ we infer that
\[
b(f)\geq \mu.
\]
Finally, letting $\mu\to L_f$ we obtain
\begin{equation}\label{equ2}
b(f)\geq L_f.
\end{equation}
By \eqref{equ1} and \eqref{equ2}, we have $b(f)=L_f$ and the proof of (1) is complete.

We now proceed with the proof of (3). Set
\[S_f:=\sup\{\lambda>0:\ \lim_{r\to 1}(1-r)M(r)^{\lambda}=0\}.\]
Let $p>0$, $\alpha>-1$ and assume that $f\in A_{\alpha}^p(\D)$ is conformal. Then by Corollary \ref{g1},
\[
\lim_{r\to 1}(1-r)M(r)^{\frac{p}{\alpha+2}}=0
\]
and thus $\frac{p}{\alpha+2}\leq S_f$. Taking the supremum, we have
\begin{equation}\label{equ3}
b(f)\leq S_f.
\end{equation}
We would now like to show that equality holds in \eqref{equ3}. Suppose this is not the case. Then there is an $\epsilon>0$ so that $b(f)+\epsilon<S_f$. By the definition of $S_f$, it follows that 
\[
\lim_{r\to 1}(1-r)M(r)^{b(f)+\epsilon}=0,
\]
which in turn implies that 
\[
M(r)^{b(f)+\frac{\epsilon}{2}}\leq (1-r)^{-\frac{b(f)+\frac{\epsilon}{2}}{b(f)+\epsilon}},
\]
for $r$ close to 1. Therefore,
if $\delta \in (0,1)$ is sufficiently close to $1$, then
\[
\int_{\delta}^{1}M(r)^{b(f)+\frac{\epsilon}{2}}dr<\infty.
\]
By Theorem \ref{bgp}, we deduce that $f\in H(\D)^{b(f)+\frac{\epsilon}{2}}$. Since $h(f)=b(f)$, this is evidently a contradiction. Therefore, $b(f)=S_f$.

The proofs of (2) and (4) are carried out exactly in the same manner, the only differences being that we use (2), instead of (1), of Corollary \ref{g1} and \eqref{geom29} instead of \eqref{growth}.
\qed
\section{Example}\label{exa}
Let $\mathcal{U}$ denote the class of all conformal mappings defined on $\D$. By the second part of Corollary \ref{inclu}, 
\begin{equation}\label{sp}
H^q(\D)\cap\mathcal{U}\subset A_{\alpha}^p(\D)\cap\mathcal{U},
\end{equation}
for any $p>0$ and $\alpha>-1$ satisfying $0<\frac{p}{\alpha+2}\leq q$.  For $\alpha=0$ and $p=2q$, we obtain the inclusion $H^q(\D)\cap\mathcal{U}\subset A^{2q}(\D)\cap\mathcal{U}$, for any $q>0$. A stronger statement is actually true, namely $H^q(\D)\subset A^{2q}(\D)$, and it follows from a theorem of Hardy and Littlewood \cite{HL}. The authors in \cite[p. 852]{Ba2} prove by means of explicit functions that, for any $q>0$, there exists a conformal mapping in $A^{2q}(\D)\setminus H^q(\D)$. Next, we exhibit a conformal mapping $f$ having the following properties:
\begin{enumerate}
\item $f\in\mathcal{U}$,
\item $f\in A_1^3(\D)$,
\item $f\notin A^2_0(\D)$,
\item $f\notin H^1(\D)$.
\end{enumerate}
Properties (2) and (4) show that the inclusion \eqref{sp} is, in general, strict. Moreover, properties (2) and (3) prove that if we consider $\alpha,\alpha'>-1$ and $p,p'>0$ such that $\frac{p}{\alpha+2}=\frac{p'}{\alpha'+2}$, then the equality 
\[A_{\alpha}^p(\D)\cap\mathcal{U}=A_{\alpha'}^{p'}(\D)\cap\mathcal{U},\]
is not always true.

We note that our function is very similar to the functions considered in \cite{Ba2} and the proof is along the same lines.

\begin{proof}
Consider the function
\[
f(z)=\frac{1}{(1-z)\sqrt{\log\frac{2e}{1-z}}},\ z\in\D.
\]
The mapping $\frac{2e}{1-z}$ maps $\D$ conformally onto $\{\Re z>e\}$ and we can define $\log z$ to be analytic there by choosing the argument in $(-\pi/2,\pi/2)$. The image, under this logarithm, of $\{\Re z>e\}$ is a simply connected subregion of $\{\Re z>1\}$ and thus by choosing the argument for the square root to be in $(-\pi/2,\pi/2)$ again, we see that $f$ is well defined and analytic in $\D$.

First, we show that 
\begin{equation}\label{maxim}
M_f(r)=\frac{1}{(1-r)\sqrt{\log\frac{2e}{1-r}}}.
\end{equation}
Note that the function $\psi(x)=x\sqrt{\log\frac{2e}{x}}$ is increasing for $x\in(0,2)$. Therefore, for $|z|=r\in(0,1)$,
\[
\bigg\lvert(1-z)\sqrt{\log\frac{2e}{1-z}}\bigg\rvert\geq  |1-z|\sqrt{\log\frac{2e}{|1-z|}}\geq (1-r)\sqrt{\log\frac{2e}{1-r}}.
\]
This implies that
\[
M_f(r)\leq \frac{1}{(1-r)\sqrt{\log\frac{2e}{1-r}}}.
\]
However, the right hand side of this estimate equals $f(r)$ and thus we have \eqref{maxim}. Suppose for a moment that (1) holds. If we choose $p>0$ and $\alpha \ge -1$ such that $\alpha+2-p=0$, then,  by \eqref{maxim},
\[\int_{0}^{1} (1-r)^{\alpha+1}M_f^p(r) dr=\int_0^1 {\frac{1}{(1-r)\left( \log\frac{2e}{1-r}\right)^{{p}/{2}}}dr}=\int_{\log (2e)}^{\infty} \frac{1}{y^{p/2}}dy,\]
where we made the change of variable $\log\frac{2e}{1-r}=y$. This and Theorem \ref{bgp} imply that (2), (3), and (4) hold as well. Therefore, it only remains to prove (1). To do this, we need the following lemma.
\begin{lemma}\label{lemap}
For $z\in\D$, 
\[
\Re\left[\left(\log\frac{2e}{1-z}\right)^{-1/2}-\frac{1}{2}\left(\log\frac{2e}{1-z}\right)^{-3/2}\right]>0.
\]
\end{lemma}
\begin{proof}
Recall that the mapping $\frac{2e}{1-z}$ maps $\D$ conformally onto $\{\Re z>e\}$ and the mapping $\log z$, as chosen above, maps $\{\Re z>e\}$ conformally onto a simply connected subregion of $\{\Re z>1\}$. It is not hard to check that the mapping $z^{-1/2}$ maps $\{\Re z>1\}$ conformally onto a simply connected subregion of the slice $\{z\in\D:\ |\arg z|<\pi/4\}$. It follows that the composition
\[
g(z)=z^{-1/2}\circ\log z\circ \frac{2e}{1-z}=\left(\log\frac{2e}{1-z}\right)^{-1/2}
\]
maps $\D$ conformally onto a simply connected domain $\Omega\subset\{z\in\D:\ |\arg z|<\pi/4\}$. Now, we show that $\Re\left(z-\frac{1}{2}z^3\right)\geq 0$, for $z\in\partial\{z\in\D:\ |\arg z|<\pi/4\}$.

If $z=re^{i\pi/4}$, $r\in[0,1]$, then 
\[
\Re \left(re^{i\pi/4}-\frac{1}{2}r^3e^{3i\pi/4}\right)=\frac{\sqrt{2}}{2}r(1+\frac{r^2}{2})\geq 0.
\]
By symmetry, the same calculation is valid for $z=re^{-i\pi/4}$. If $z=e^{it}$, for $|t|<\pi/4$, then 
\[
\Re\left(e^{it}-\frac{1}{2}e^{3it}\right)=\cos t-\frac{1}{2}\cos (3t)\geq \frac{\sqrt{2}}{2}-\frac{1}{2}\cos 3t>0.
\]
Hence, $\Re\left(z-\frac{1}{2}z^3\right)\geq 0$, for $z\in\partial\{z\in\D:\ |\arg z|<\pi/4\}$ and by the maximum principle, the inequality is strict for $z\in \{z\in\D:\ |\arg z|<\pi/4\}$. Since $\Omega\subset \{z\in\D:\ |\arg z|<\pi/4\}$, we have
$\Re\left(z-\frac{1}{2}z^3\right)>0$ in $\Omega$ and therefore
\[
\Re\left[\left(z-\frac{1}{2}z^3\right)\circ g(z)\right]=\Re\left[g(z)-\frac{1}{2}g(z)^3\right]>0,\ z\in\D.
\]
\end{proof}

We can now prove (1). Upon differentiating $f$, we find
\[
f'(z)=(1-z)^{-2}\left[\left(\log\frac{2e}{1-z}\right)^{-1/2}-\frac{1}{2}\left(\log\frac{2e}{1-z}\right)^{-3/2}\right].
\]
By Lemma \ref{lemap}, if $z\in\D$,
\begin{align*}
\Re (1-z)^2f'(z)&=\Re \frac{zf'(z)}{K(z)}\\
&=\Re \left[\left(\log\frac{2e}{1-z}\right)^{-1/2}-\frac{1}{2}\left(\log\frac{2e}{1-z}\right)^{-3/2}\right]>0,
\end{align*}
where $K(z)=\frac{z}{(1-z)^2}$ is the Koebe function. Using the terminology of Sections 2.2 and 2.3 of \cite{Pom0}, it follows that, since $K$ is starlike, the function $f(z)-f(0)$ is close-to-convex and thus by \cite[Theorem 2.11, p. 51]{Pom0}, $f\in\mathcal{U}$.
\end{proof}

\end{document}